\documentclass{dcds}

\usepackage{amsmath,amssymb}
\usepackage[utf8]{inputenc}
\usepackage[english]{babel}
\usepackage{csquotes,paralist,psfrag, subfigure,relsize,graphics,epsfig,graphicx,epstopdf}
 \usepackage[colorlinks=true]{hyperref}
\hypersetup{urlcolor=blue, citecolor=red}

\def\currentvolume{}
 \def\currentissue{}
  \def\currentyear{}
   \def\currentmonth{}
    \def\ppages{X--XX}
     \def\DOI{}
\everymath{\displaystyle}

\newtheorem{thm}{Theorem}[section]
\newtheorem{clly}{Corollary}[section]

\newtheorem{lema}{Lemma}[section]
\newtheorem{prop}{Proposition}[section]
\newtheorem{conjecture}{Conjecture}[section]

\theoremstyle{definition}
\newtheorem{defi}{Definition}[section]
\newtheorem{rk}{Remark}[section]

\newcommand{\N}{\mbox{$\mathbb{N}$}}
\newcommand{\T}{\mbox{$\mathbb{T}$}}
\newcommand{\Z}{\mbox{$Z\!\!\!Z$}}

\newcommand{\R}{\mbox{$I\!\!R$}}

\numberwithin{equation}{section}

\makeatother

\begin{document}

\title[ Robust transitivity of singular endomorphisms ]{ Robustly transitive maps of $\T^n$ with high dimensional central spaces.}

\author[Juan Carlos Morelli]{}

\subjclass{Primary: 37C20; Secondary: 57R45, 57N16.}
 \keywords{Transitivity, stability, robustness, high dimension, central space.}

\email{jmorelli@fing.edu.uy }

\maketitle

\centerline{\scshape  Juan Carlos Morelli$^*$}
\medskip
{\footnotesize
 \centerline{Universidad de La Rep\'ublica. Facultad de Ingenieria. IMERL.}
   \centerline{ Julio Herrera y Reissig 565. C.P. 11300.}
   \centerline{ Montevideo, Uruguay.}}

\bigskip

 \centerline{(Communicated by )}

\begin{abstract}

 Given any positive integers $n$, $m$ and $k$ such that $n=m+k$, we exhibit a $C^1$ robustly transitive endomorphism $F$ of $\mathbb{T}^n$ such that $F$ is in the isotopy class of $f \times Id$ where $f$ is an expanding map of $\mathbb{T}^m$ and $Id$ is the identity of $\mathbb{T}^k$. Furthermore, for small $k$, the map $F$ is in fact a perturbation of $f \times Id$. The construction allows for $F$ to be both, either regular or singular, endomorphism of $\T^n$.

\end{abstract}

\section{Introduction}

Let $X$ be a real (riemannian) manifold and an induced discrete dynamical system over $X$ by a continuous map $f: X \to X$. Many of the features of interest displayed by the dynamical system depend on the properties displayed by $f$, and whichever property $f$ possesses is said to be \textit{robust} if all maps sufficiently close to $f$ also posses that same property. Our article addresses robustness of transitivity, meaning by \textit{transitive} the existence of a forward dense orbit for some point $x$ in $X$.\\
Robust transitivity has been an active topic for about the last fifty years. The first results showed that for \textsl{diffeomorphisms} it is needed to have at least weak hyperbolicity for $C^1$ robust transitivity (see \cite{m2}, \cite{bdp}), while \textit{Anosov on $\T^2$} is an equivalent condition in the surfaces' setting. Concerning \textsl{regular endomorphisms} (non-injective and empty critical set), weak forms of hyperbolicity are shown to be necessary for robust transitivity, while sufficient conditions are known for the $n$-torus as phase space (see \cite{and}, \cite{lp}, \cite{p}).  \\
Robust transitivity of \textsl{singular endomorphisms} (non-empty critical set) has been the least studied setting, only in the last decade the first such examples were exhibited (see \cite{br}, \cite{ilp}) on the $2$-torus. The setting became active recently. In 2019 was shown that weak hyperbolicity is needed for $C^1$ robust transitivity of singular maps on any surface (see \cite{lr1}, \cite{lr2}). The higher dimensional setting has been approached even less: on the one hand only two examples of such maps have been recently exhibited (see \cite{mo}, \cite{mo2}) while on the other, the first necessary condition for their existence has been stated very recently (see \cite{lppr}), being the {\it State of the Art} in the subject.\\

The aim of our article is to improve the construction given by \cite{mo2}. There appears the first example of a robustly transitive map in the $C^1$ topology displaying persistent critical points in high dimension. The construction is carried on such that the central space\footnote{by {\it central space} we mean a section of the tangent bundle which is invariant by the differential map and its vectors are not uniformly expanded nor uniformly contracted. In particular, eigenspaces associated to eigenvalues of modulus one are center spaces.} is one dimensional, which is a classical approach in the literature. The problem of dealing with central spaces of larger dimensions is well known and it corresponds to the loss of control within these spaces. Moreover, important open problems such as the Pugh-Shub conjecture, the Shub-Wilkinson conjecture, etc. have found already partial answers in the {\it one-dimensional center space} set up, while the {\it two-or-larger-dimensional center space} set up remains open (see \cite{ps}, \cite{rru}, \cite{bw}, \cite{sw}). Here, we adjust the construction carried on in \cite{mo2}, showing that the central space does not need be one dimensional but can in fact be as large dimensional as desired.\\

The main result of this work is stated in Theorem \ref{mainendo}. Afterwards, the map satisfying Theorem \ref{mainendo} can be perturbed to a new map admitting robust singularities, yielding the result given by Theorem \ref{main}. In the next section precise definitions appear.

\begin{thm}\label{mainendo}
  Let $n$ and $k$ be natural numbers such that $n \geq 2$ and $1 \leq k \leq n-1$ and let $\mathbb{T}^n$ be the $n$-torus. Then, there exists a smooth map $F: \mathbb{T}^n \to \mathbb{T}^n$ satisfying:
  \begin{itemize}
    \item $F$ is $C^1$ robustly transitive.
    \item The central space of $F$ is $k$-dimensional.
  \end{itemize}
\end{thm}

\begin{thm}\label{main}
  Let $n$ and $k$ be natural numbers such that $n \geq 2$ and $1 \leq k \leq n-1$ and let $\mathbb{T}^n$ be the $n$-torus. Then, there exists a smooth map $F: \mathbb{T}^n \to \mathbb{T}^n$ satisfying:
  \begin{itemize}
    \item $F$ is $C^1$ robustly transitive.
     \item The central space of $F$ is $k$-dimensional.
   \item $F$ is robustly singular.
  \end{itemize}
\end{thm}

    \subsection{Sketch of the Construction.} Start from an endomorphism of $\mathbb{T}^n$ induced by a diagonal matrix with integer coefficients, where the first $(n-k)$ entries are large positive integers (yielding one dimensional strong unstable directions) and the last $k$ entries equal to one (central directions). Perturb the induced linear map to add a {\it blending region}. Roughly speaking, it consists of several {\it slices} of of $\mathbb{T}^n$ where two mechanisms take place: on the one hand, the restriction of the perturbation to some of these slices configures a set of maps that {\it shrink} open sets; and its restriction to the remaining slices configure a robustly minimal IFS. These two mechanisms acting together yield topological mixing inside the slices. Then, a field of unstable cones  transversal to the slices replays the behavior over the whole torus, getting the transitivity from local dynamics over this region. Finally, for the singular version, a set of critical points is introduced artificially with a second perturbation far away from the blending region so that the transitivity property remains unaffected. All the construction is done in a robust way.\\

    The author emphasizes that the contents to follow are an improvement to the construction appearing in \cite{mo2}, where the case $n \geq 2$ and $k=1$ is solved. The proofs to some of the claims in our Lemmas and Theorems are, consequently, also inspired by \cite{mo2}. If readers wish to get a hollower approach to our construction they are gently invited to get in touch with the article.\\

	It is the hope of the author that the ideas and techniques involved in the proofs up next could be of use to colleagues taking on the higher dimensional context of central spaces.

\section{Preliminaries}

We recall some definitions here, for more insight on geometrical or dynamical background the readers can refer themselves to \cite{gg} or \cite{kh}.

Let $M$ be a differentiable manifold of dimension $m$ and $f:M \rightarrow M$ a differentiable endomorphism and $D_xf:T_xM \to T_xM$ its differential map (or its Jacobian matrix, indistinctively). The \textbf{critical set} of $f$ is $S_f:=\{ x \in X : det(D_xf)=0\}$. The map $f$ is said to be \textbf{singular} if its critical set is nonempty, and it is said to be \textbf{robustly singular} if there exists a neighborhood $\mathcal{U}_{f}$ of $f$ in the $C^1$ topology such that the critical set $S_g$ is nonempty for all $ g \in \mathcal{U}_{f}$. The \textbf{orbit} of $x \in M$ is $\mathcal{O}(x)=\{ f^n(x) , n \in \N \}$. The map $f$ is said to be \textbf{transitive} if there exists a point $x \in M$ such that $ \overline{{\mathcal{O}(x)}}=M$, and it is said that $f$ is $C^k$-\textbf{robustly transitive} if there exists a neighborhood $\mathcal{U}_f$ of $f$ in the $C^k$ topology such that $g$ is transitive for all $ g$ belonging to $\mathcal{U}_f$.\\
The proposition ahead is well known, useful and easy to prove.

\begin{prop}\label{equi}
  If $f$ is continuous then are equivalent:
  \begin{enumerate}
    \item  $f$ is transitive.
    \item  For all $U, V$ open sets in $M$, exists $n \in \N$ such that $ f^n(U) \cap V \neq \emptyset$.
    \item  There exists a residual set $R$ (countable intersection of open and dense sets) such that for all points $  x \in R: \overline{\mathcal{O}(x)}=M$. \hfill $\Box$
  \end{enumerate}
\end{prop}


\subsection{Normal Hyperbolicity.}

We continue defining normally hyperbolic submanifolds in the sense of \cite{bb}. These kind of submanifolds for a given map are persistently invariant under perturbation, it allows defining dynamical systems within them. This will be the main usage we will make of them ahead in the paper. Their formal definition is as follows.

Let $f:M \to M$ be a $C^1$ diffeomorphism, $N \subset M$ a $C^1$ closed submanifold such that $f(N)=N$ (we say that $N$ is \textit{invariant}).
\begin{defi}
  We say that $f$ is \textit{Normally Hyperbolic at $N$} if there exists a splitting of the the tangent bundle of $M$ over $N $ into three $Df$-invariant subbundles such that $TM_{|N} = E^s \oplus E^u \oplus TN$ and a constant $0 < \lambda < 1$ such that for all $x \in \mathbb{N}$ the following hold:
  \begin{itemize}
    \item $||D_xf_{|E^s_x}|| < \lambda$, $||(D_xf)^{-1}_{|E^u_x}|| < \lambda$,
    \item $||D_xf_{|E^s_x}||.||(D_{f(x)}f)^{-1}_{|T_{f(x)}N}|| < \lambda$,
    \item $||(D_xf)^{-1}_{|E^u_x}||.||(D_{f^{-1}(x)}f)_{|T_{f^{-1}(x)}N}|| < \lambda$.
  \end{itemize}
 \end{defi}

 The first condition implies that the behavior of the differential map $Df$ is hyperbolic over $(T_x M)_{|N} \setminus T_x N$ while the other two describe the domination property relative to stable $E^s$ and unstable $E^u$ subspaces. Our interest in these submaniolds comes from \cite[Theorem 2.1]{bb} which states:
  \begin{thm}\label{NHSP}
     Given $M,N$ and $f$ as in the definition above, there exists $\mathcal{U}_f$ a $C^1$ neighborhood of $f$ such that all $g \in  \mathcal{U}_f$ admit a $C^1$ invariant submanifold $N_g$ which is unique such that $g$ is normally hyperbolic at $N_g$. Moreover, $N$ and $N_g$ are diffeomorphic and there exists an embedding from $N$ to $N_g$ which is $C^1$ close to the canonical inclusion $i:N \to M$.
  \end{thm}

\subsection{Iterated Function Systems.}

Let $\mathcal{F},\mathcal{G}$ be two families of diffeomorphisms of $M$. Denote by $\mathcal{F} \circ \mathcal{G}:= \{f \circ g  / \quad f \in \mathcal{F}, g \in \mathcal{G}\}$; and for $k \in \N$ denote $\mathcal{F}^0=\{ Id_M\}$ and $\mathcal{F}^{k+1}=\mathcal{F}^{k} \circ \mathcal{F}$. Then, the set $\bigcup_{k=0}^{\infty}\mathcal{F}^k$ has a semigroup structure that is denoted by $\langle \mathcal{F}\rangle^+$ and said to be generated by $\mathcal{F}$. The action of the semigroup $\langle \mathcal{F}\rangle^+$ on $M$ is called the \textbf{iterated function system} associated with $\mathcal{F}$. We denote it by IFS$(\mathcal{F})$. For $x \in M$, the \textbf{orbit} of $x$ by the action of the semigroup $\langle \mathcal{F}\rangle^+$ is $ \langle \mathcal{F}\rangle^+(x)=\{f(x), f \in \langle \mathcal{F}\rangle^+ \}$. A sequence $\{ x_n, \quad n \in \N \}$ is a branch of an orbit of IFS$(\mathcal{F})$ if for every $n \in \N$ there exists $f_n \in \langle \mathcal{F}\rangle^+$ such that $f_n(x_n)=x_{n+1}$.
\begin{defi}
  An IFS$(\mathcal{F})$ is \textbf{minimal} if for every $x \in M$ the orbit $ \langle \mathcal{F}\rangle^+(x)$ has a branch that is dense on $M$. \\
  An IFS$(\mathcal{F})$ is \textbf{$C^r$ robustly minimal} if for every family $\mathcal{\hat{F}}$ of $C^r$ perturbations of $\mathcal{F}$ and every $x \in M$ the orbit $\langle \mathcal{\hat{F}}\rangle^+(x)$ has a branch that is dense on $M$.
\end{defi}

The following result appearing on \cite[Theorem A]{hn} is one of the two key mechanisms required for our construction.

\begin{thm}\label{HNT}
Every boundaryless compact manifold admits a pair of diffeomorphisms that generate a $C^1$ robustly minimal IFS.
\end{thm}

 The second key mechanism consists of {\it some way} to \textit{enlarge} open sets. We show that there exists a set $\mathcal{F}$ of $(n+1)$ maps on $\mathbb{T}^n$ that are \textit{almost contractions} and have a bounded $C^1$ distance to the identity, such that the preorbit by $\langle \mathcal{F} \rangle$ of any open set in $\mathbb{T}^n$ contains a big ball for all sufficiently large preiterates.\\
  For the proof, let $S^1$ be the quotient of $[-1,1]$ under $-1\sim 1$ and $\mathbb{T}^n= \prod_{i=1}^{n}S^1 $.\\
  Recall that if $X$ denotes a compact subset of $\mathbb{T}^n$, its \textit{internal radi} is defined as $ir(X):=\max_{x \in X}\{r>0/ \quad B_k{(x,r)} \subset X \}$. If $W$ is open in $\mathbb{T}^n$, $ir(W):=ir(\overline{W})$.

\begin{lema}\label{SRM}
  Given $\alpha >0$, there exists a family $\mathcal{F}=\{g_1,...,g_{n+1}\} \subset Diff^1(\mathbb{T}^n)$ such that:
   \begin{itemize}
     \item $\left\{ x \in \mathbb{T}^n  / \mbox{ for some $i \in \{1,...,n+1\}$, $g_i$ is a contraction at $x$} \right\}=\mathbb{T}^n$.
         \item $\max_{i \in \{1,...,n+1\}}\{ ||Id -g_i||,||Id -Dg_i||\} <\alpha$,
     \item Given $W$ open in $\mathbb{T}^n$, there exist $p_0 \in \mathbb{N}$ such that for all $p \geq p_0$, the inradius of preimage $\langle \mathcal{F}\rangle^{-p}(W)$ is larger than $\frac{\sqrt{n}}{n+1}$.

   \end{itemize}
\end{lema}

\begin{figure}[ht]
\begin{center}
\includegraphics[scale=0.275]{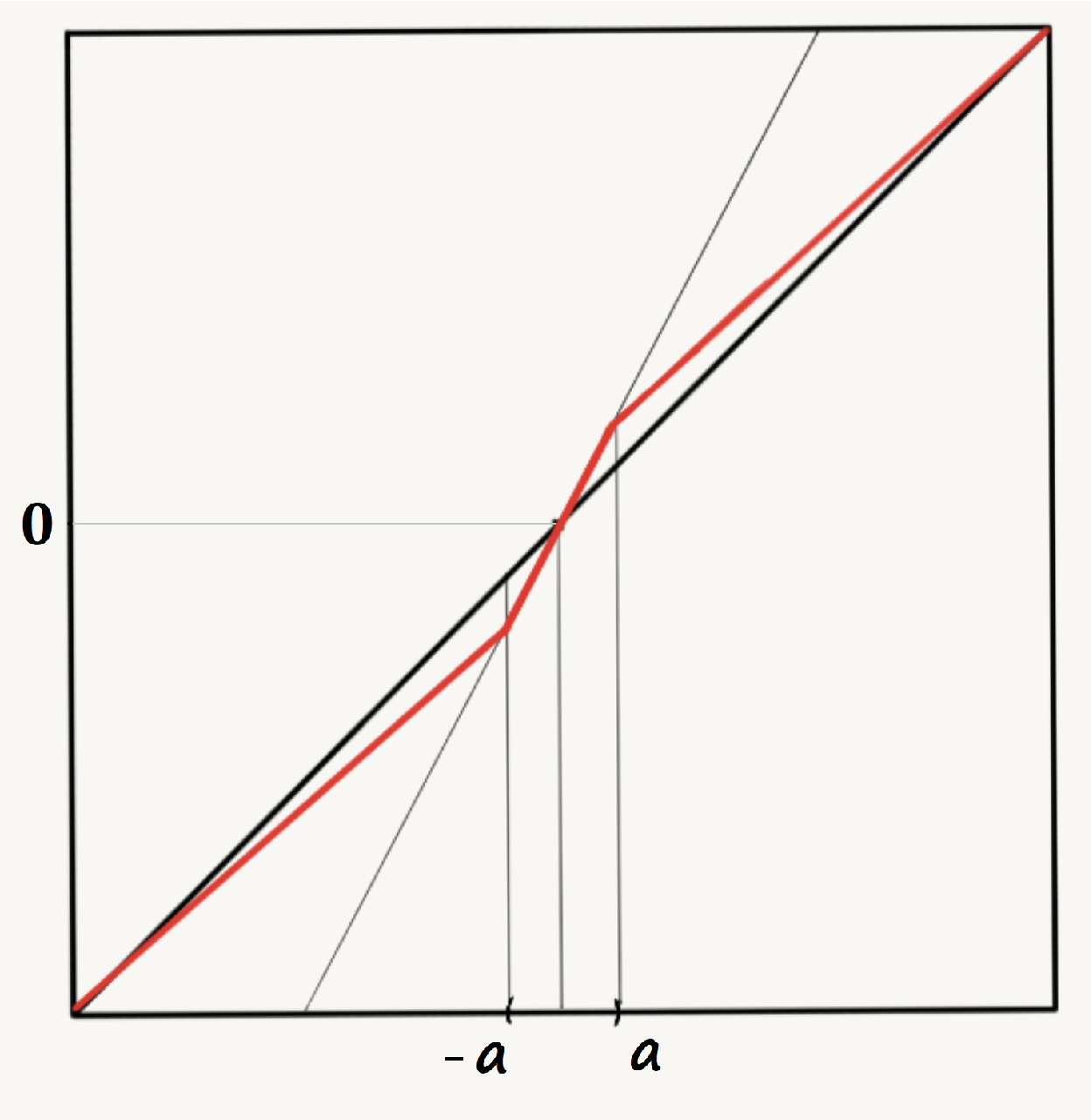}
\caption{$\tilde{g}_a:S^1 \to S^1$ is almost a contraction on $S^1$. }\label{IFSmap}
\end{center}
\end{figure}

\begin{proof}
  Let $a \in \left(0, \frac{1}{4n} \right)$ and $g_a:[-1,1] \to [-1,1]$ a real function given by $$g_a(x)=\left\{ \begin{array}{c}
  \begin{aligned}
 \left(1+ \frac{a\alpha}{2(a-1)}\right)(x+1)-1, & \mbox{\ \ \ \ if $x \in \left[-1, -a \right]$}\\
\left(1+\frac {\alpha}{ 2} \right)x, & \mbox{\ \ \ \ if $x \in \left[ -a, a \right]$} \\
\left( 1+ \frac{a\alpha}{2(a-1)}\right)(x-1)+1, & \mbox{\ \ \ \ if $x \in \left[ a, 1 \right]$}
\end{aligned}
\end{array}\right.$$

  Then $g_a$ is a continuous piece-wise linear function that descends to $S^1$ as shown in Figure \ref{IFSmap}.
  Fix $a_0 \in \left(0, \frac {1}{ 8n} \right]$ so that $||x-g_{a_0}(x)||< \frac{\alpha}{2}$.
  Let $\tilde{g}:S^1 \to S^1$ be a smooth approximation of $g_{a_0}$ such that $||\tilde{g}'||\leq 1 + \frac \alpha 2 $ and $\tilde{g}$ is a contraction on the complement of $\left( -2a_0,  2a_0 \right)$. Afterwards, let $g:\mathbb{T}^n \to \mathbb{T}^n$ be such that $g(x)=\prod_{i=1}^{n}\tilde{g}(x)$ and define a family $\mathcal{F}=\{ g_i: \mathbb{T}^n \to \mathbb{T}^n, 0 \leq i \leq n \mbox{ integer} \}$ such that $g_i(x)=g\left(x- \frac{2i}{n+1}\right)+ \frac{2i}{n+1}$. We show next that $\mathcal{F}$ is a family that satisfies the claims at the thesis of the lemma. Observe that the first three properties are straightforward from the construction. \\ To prove the last assertion, notice that for any small and connected open set $W$, since $g_i$ are almost contractions, $\langle \mathcal{F}\rangle^{-1}(W)$ will display a larger inradi than $W$ unless the canonical projection of $W$ intersects any of the (disjoint) expanding regions of the $g_i$, but it is not possible that $W$ has $n$ projections within $n+1$ disjoint intervals. Hence, for some $i$, $g_i^{-1}(W)$ contains an open set that has a larger inradi than $W$. This argument will remain to hold until $W$ becomes large enough so that it intersects two (or more) of the expanding disjoint regions of the $g_i$, but in this case it has to be $ir \left( \langle \mathcal{F}\rangle^{-p_0}(W) \right) \geq \frac{2\sqrt{n}}{n+1}-2a_0$ for some $p_0$, which yields  $ir \left( \langle \mathcal{F}\rangle^{-p}(W) \right) > \frac{\sqrt{n}}{n+1}$ for all $p \geq p_0$.\end{proof}
\begin{rk}
  The family $\mathcal{F}$ is $C^1$-close to the Identity map.
\end{rk}

The union of the sets given by Theorem \ref{HNT} and Lemma \ref{SRM} yield the following:
\begin{clly}\label{cllySRM}
  There exists a family $\mathcal{F}=\{g_1,...,g_{n+3}\} \subset Diff^1(\mathbb{T}^n)$ such that:
   \begin{itemize}
     \item IFS$(\mathcal{F})$ is $C^1$ robustly minimal,
    \item Given $W$ open in $\mathbb{T}^n$, there exist $p_0 \in \mathbb{N}$ such that for all $p \geq p_0$, the preimage $\langle \mathcal{F}\rangle^{-p}(W)$ contains an open set of diameter larger than $\frac{\sqrt{n}}{n+1}$.\hfill $\Box$
   \end{itemize}
\end{clly}

Having stated all the preliminary facts needed to construct the example map satisfying the claim at the thesis of Theorem \ref{main}, we proceed to it in two steps. But first, an overview of the whole construction in a simpler setting is given in Section 3. Next up, in Section 4, we define an endomorphism of $\mathbb{T}^n$ (named $f$) that is $C^1$ robustly transitive. To do so, we use the result given by Corollary \ref{cllySRM} to create a robust blending region for $f$ supported on a strict subset $X$ of $\mathbb{T}^n$, aided by a field of unstable cones.\\
Once this is achieved, we move on into Section 5 where the second step of the construction takes place by introducing critical points artificially inside the complement of $X$ in $\mathbb{T}^n$. The surgery is performed such that the critical points existence is robust and the blending region is unaffected, resulting in a new map (named $F$) that satisfies the claim at Theorem \ref{main}.

\section{A low dimensional example.}

Let us begin giving a full proof of Theorem \ref{main} when $n=3$ and $k=2$. After this simpler approach, the reader will grasp the ideas supporting the construction. The general case is proved with the same ideas in a more complex environment. The main Theorem adopts the following form:

\begin{thm}\label{lowmain}
  There exists a smooth map $F: \mathbb{T}^3 \to \mathbb{T}^3$ satisfying:
  \begin{itemize}
    \item $F$ is $C^1$ robustly transitive.
    \item $F$ is robustly singular.
    \item The central space of $F$ is $2$-dimensional.
  \end{itemize}
\end{thm}

\subsection{Construction of $f$.}

For the rest of the section, let the quotient $(\mathbb{R} / [-1,1])^3$ be $\mathbb{T}^3=S^1 \times \mathbb{T}^{2}$  and endow it with the standard (euclidean) riemannian metric.\\
Start fixing $\mathcal{F}_1=\{m_3,m_4\}$ the family given by Theorem \ref{HNT} for the second factor $\mathbb{T}^{2}$.

\begin{rk}\label{lowF1F2} \phantom{.}
\begin{enumerate}
  \item $\max \{ ||Id_{\mathbb{T}^2}-m_3||, ||Id_{\mathbb{T}^2}-m_4||\} =: M \in \mathbb{R}$ exists.
\item As is stated in \cite{hn}, $\mathcal{F}_1$ can be constructed such that for given small $\delta >0$ the jacobians satisfy $( ||Dm_3|| -1)^2+( ||Dm_4||-1)^2<\delta$.
    \end{enumerate}
\end{rk}

Once $\mathcal{F}_1$ is set, define $m_5$ another Morse-Smale diffeomorphism of $\T^2$ such that  $||Id_{\mathbb{T}^2}-m_4||\leq M$ with an only attractor contained in the interior of basin of attraction of the attraction of $m_3$ and such that the stable manifolds of the saddles of $m_3$ insersect transversally at most in one point with the stable manifolds of the saddles of $m_5$ (it is enough taking $m_5(x)=m_3(x-p)+p$ where $p$ is contained in the interior of the basin of $0$ in $m_3$). Afterwards, fix two real numbers $r$ and $\kappa$ such that $0< r < \frac{1}{20}$ and $0 < \kappa <1$ ($\kappa$ will be useful to prove that the sought map admits unstable cones), and let $\widehat{A} \in \mathcal{M}_3(\Z)$ be the diagonal matrix below, with a large integer $\lambda$  in the first column satisfying $\lambda >> \left\{ \frac{1}{r}, \frac{80M}{r \kappa}\right\}$, where $M$ is defined in 1 at Remark \ref{lowF1F2}. The other two entries in $A$ are equal to $1$.

 \begin{equation}\label{lowmatrizAgorro}
   \widehat{A}= \left(\begin{array}{ccc}
 \lambda  & 0 & 0\\
0  & 1 & 0 \\
0  & 0 &1\\
\end{array}
\right).
 \end{equation}

The matrix $\widehat{A}$ induces a regular endomorphism $A$ on the torus defined by
\begin{equation}\label{lowendoA}
  A:\T^3\rightarrow\T^3 / \quad A(x_1,x_2,x_3)= (\lambda x_1,x_2, x_3).
\end{equation}

\begin{rk}\label{lowrklambda}\phantom{.}
  \begin{enumerate}
    \item  Since $r <\frac{1}{20}$, then $\lambda >20$ hence $A$ displays more than 5 fixed points. In particular, $Fix(A)=\left\{ \left(\frac{2i}{\lambda -1},p_2,p_3 \right) \in \mathbb{T}^3, i \in \mathbb{Z} \cap [0, \lambda -1) \right\}$.
         \item The choice of a very large value of $\lambda$ yields a very strong unstable space.
        \item The map $A$ is {\it modulo 2}. Even when we do not state it explicitly, it applies for all maps of $\T^3$ defined along this section.
    \item The construction is possible for any $\lambda_0 \in \Z$ such that $|\lambda_0|>1$ since there would be a power of $\widehat{A}$ such that the first entry would be larger than $\lambda$. It follows that the construction holds for any linear map in the isotopy class of maps with one eigenvalue of modulus larger than one.

  \end{enumerate}
\end{rk}

Let $r$ be the real number fixed at the beginning of the section and define 5 disjoint subsets (shrink $r$ if necessary) of $S^1$ by $K_i=\left[\frac{2i}{\lambda -1}-r,\frac{2i}{\lambda -1}+r \right]$ and let $K= \bigcup_{0 \leq i \leq 5} K_i $. Define next $\tilde{K}_{i}=\left[\frac{2i}{\lambda -1}-2r,\frac{2i}{\lambda -1}+2r \right]$ and let $\tilde{K}= \bigcup_{0 \leq i \leq 5}\tilde{K} $.\\

For each $i$, define a smooth map $u_i: \R \rightarrow \R$ such that $u_i{|K_i}=1$ and $u_i{|\tilde{K}_i^c}=0$ (see Figure \ref{lowgraficou}) and let $u=\sum_{i=0}^{5}u_i$. Observe that $||u'||:=\max \{ |u'(x)|, x \in \R \}< \frac{2}{r}$.

\begin{figure}[ht]
\begin{center}
\includegraphics[scale=0.5]{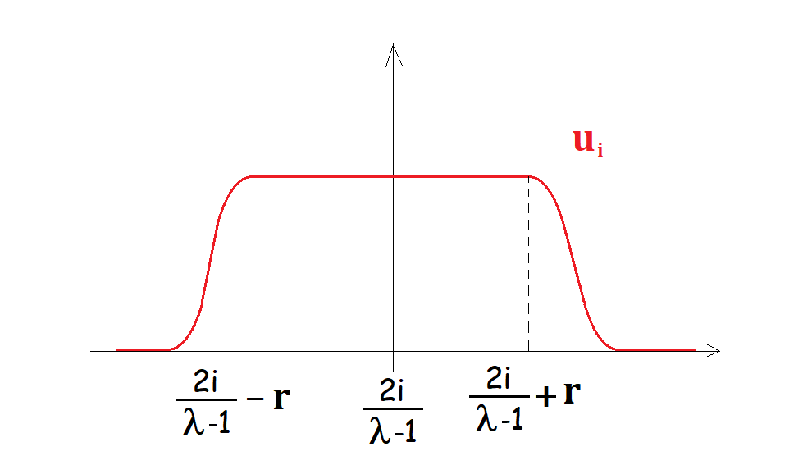}
\caption{Graph of $u_i$.}\label{lowgraficou}
\end{center}
\end{figure}

Finally, let $\mathcal{F}_2=\{g_0,g_1,g_2\}$ be the family given by Lemma \ref{SRM} for the second factor $\T^{2}$, satisfying the properties claimed in the Lemma for $\alpha \leq \frac{\kappa}{40||u'||}$. The family from Corollary \ref{cllySRM} becomes $\mathcal{F}=\mathcal{F}_1 \cup \mathcal{F}_2$. Define \begin{equation}\label{lowmainmapita}\hat{f}:\tilde{K} \times \mathbb{T}^{2} \rightarrow \mathbb{T}^3  /\quad  \hat{f}(x,y,z)=
\left\{
                                \begin{array}{c}
                                \begin{aligned}
(\lambda x , g_i(y,z)) & \mbox{  if $x \in \tilde{K}_i $, $i \in \{0,1,2\}$} \\
(\lambda x , m_i(y,z)) & \mbox{  if $x \in \tilde{K}_i $, $i \in \{3,4,5\}$}
\end{aligned}
\end{array}\right.\end{equation}

and extend $\hat{f}$ to $f: \mathbb{T}^3 \rightarrow \mathbb{T}^3$ given by \begin{equation}\label{lowmainmap}
 f(x,y,z)=\left\{
                                \begin{array}{c}
                                \begin{aligned}
u(x).\hat{f}(x,y,z)+(1-u(x)).A(x,y,z) & \mbox{ if $x \in \tilde{K}$} \\
A(x,y,z) & \mbox{ if $x \notin \tilde{K}$}
\end{aligned}
\end{array}\right.
\end{equation}

\begin{rk}\label{lowrkdinf}

 The following properties are straightforward to check:
\begin{enumerate}
  \item If $\hat{f}(x,y,z)=(\lambda x,\hat{f}_2(y,z))$, then $f(x,y,z)=(\lambda x,u(x).\hat{f}_2(y,z)+(1-u(x)).(y,z))$.
   \item Since $\|Dg_i\| \leq \frac 3 2 $ for all $i$, combined with Remark \ref{lowF1F2} gives $\|D\hat{f}_2\|<2$.
  \item By construction of $f$, $\| Id -\hat{f}_2\| \leq \max\left\{M, \frac{\kappa}{40||u'||}\right\}$.
  \item The restriction $f_{|\left(K \times \T^{2}\right)}=\hat{f}$.
  \item The restriction $f_{{| \left( \tilde{K} \times \T^{2}\right)}^c}=A$.
\item For all $i$, $f(K_i \times \T^{2}) = \mathbb{T}^3$ so then $K \times \mathbb{T}^2$ is a proto-blender in $\mathbb{T}^3$.
 \item For all $i$,  $\left(\frac{2i}{\lambda -1},1+\frac{2i}{\lambda -1},1+\frac{2i}{\lambda -1}\right)$ is a saddle fixed point. 
\item For all $i$,  $\lbrace\frac{2i}{\lambda -1}\rbrace \times \mathbb{T}^2$ is a normally hyperbolic submanifold. 
 \item The point $z:=(0,1,1)$ is a saddle of $f$ while $(0,0,0)$ is a repeller so then $f$ is not uniformly hyperbolic. 
\item The local unstable manifold at $z$ is given by $W^u_{loc}(z)=(-\tau,\tau)\times{(1,1)}$. The choice $\tau \geq \frac {1}{10}$ is in order to have $f(W^u_{loc}(z)) \supset S^1 \times{(1,1)}$ robustly.
\item The local stable manifold at $z$ is given by $W^s_{loc}(z)=\lbrace 0 \rbrace \times (1-\tau,1+\tau)^2$.

\end{enumerate}
\end{rk}

\subsubsection{Dynamics of $f$.}

We prove in this section that $f$ is weakly hyperbolic. It follows from the evident fact that $f$ has a very strongly dominant expanding subspace along the first coordinate. It follows that there exists a family of unstable cones for $f$ in its direction. We make a pause here to check the existence of such an unstable cone field for $f$.\\
Recall that for $x \in M$, we call  \textbf{cone} of parameter $a$, index $n-k$ and vertex $x$  to  $$C^{u}_{a} (x)=\left\{ (v_{1},...,v_{n}) \in T_x M / \frac{\Vert (v_{k+1},...,v_{n}) \Vert}{\Vert (v_{1}, v_2,...,v_k) \Vert}<a \right\}$$ and that $f$ admits an \textbf{unstable cone} of parameter $a$ and vertex $x \in M$ if there exists $C^u_a(x) \subset T_xM$ such that $\overline{D_x f(C^u_a(x)) }\setminus \{0\} \subset C^u_a(f(x))$.

\begin{lema}\label{lowconosf}
  The map $f$ defined by Equation (\ref{lowmainmap}) admits an unstable cone of parameter $\kappa$, index 2 and vertex $(x,y,z)$ at every $(x,y,z) \in \T^3$.
\end{lema}
\begin{proof}

The calculations only need to be done over $\tilde{K} \times \mathbb{T}^2$, since in its compliment $f=A$. For the proof, let $\hat{f}_2(y,z)=(p(y,z),q(y,z))$. The differential of $f$ at $(x,y,z) \in \tilde{K} \times \mathbb{T}^2$ is given by $$
                                                D_{(x,y)}f=\left( \begin{array}{ccc}
                                                   \lambda & 0 & 0 \\
                                                   u'(x).(p(y,z)-y) & u(x).\frac{\partial p}{\partial y}(y,z)+1-u(x)  & u(x).\frac{\partial q}{\partial y}(y,z)\\
                                                   u'(x).(q(y,z)-z) & u(x).\frac{\partial p}{\partial z}(y,z) &  u(x).\frac{\partial q}{\partial z}(y,z)+1-u(x) 
                                                 \end{array}\right).
                                               $$ For all $(v_1,v_2,v_3) \in T_{(x,y,z)}\T^3$, the differential $D_{(x,y,z)}f(v_1,v_2,v_3)$ is given by
 $$\left( \begin{array}{c}
                                                   \lambda v_1 \\
                                                   \left[u'(x).(p(y,z)-y)\right]v_1 + \left[u(x).\frac{\partial p}{\partial y}(y,z)+1-u(x)\right]v_2 +\left[u(x).\frac{\partial q}{\partial y}(y,z)\right]v_3\\
                                                   \left[u'(x).(q(y,z)-z)\right]v_1 + \left[u(x).\frac{\partial p}{\partial z}(y,z) \right]v_2 +\left[u(x).\frac{\partial q}{\partial z}(y,z)+1-u(x) \right]v_3
                                                 \end{array}\right).$$

Consider all $(v_1,v_2,v_3)$ in $C^u_\kappa(x,y,z)$ and let $(w_1,w_2,w_3):=Df_{(x,y,z)}(v_1,v_2,v_3)$, we prove that $C^u_\kappa$ is unstable by computing $\frac{||(w_2,w_3)||}{|w_1|}$. After a straightforward calculation the reader can verify, along with the facts listed below, that $\frac{||(w_2,w_3)||}{|w_1|} \leq \frac{|w_2|+|w_3|}{||w_1||} \leq 2.\frac{|\frac{\kappa}{40}v_1+2v_2+2v_3|}{|\lambda v_1|}+\frac{|v_2|+|v_3|}{|\lambda v_1|} \leq \frac{\kappa}{20\lambda}+\frac{10\kappa}{\lambda}<\frac{11\kappa}{\lambda}<\kappa$.
\begin{itemize}
\item $ (v_1,v_2,v_3) \in C^{u}_{\kappa} (x,y,z)$,
\item $\max_{(y,z) \in \T^2}  \{ ||p(y,z)-y||,||q(y,z)-z|| \} \leq ||\hat{f}_2-Id||$
\item $|u'(x)|.||\hat{f}_2(y,z)-(y,z)|| \leq \frac{\kappa}{40}$ by Remark \ref{lowrkdinf} and definition of $\lambda$ and $\mathcal{F}_2$,
\item $\max \left\{ ||p'||,||q'|| \right\} \leq ||D\hat{f}_2||<2$ by Remark \ref{lowrkdinf},
\item $ \| Id\| \leq 2 $,
\item $ \max_{x \in \R} \{|u(x)|,|1-u(x)|\}\leq 1$,
\item $\lambda >20$. \end{itemize} \end{proof}

\begin{rk}\label{iguales}
Notice how 

$$ D_{(x,y)}f=\left( \begin{array}{ccc}
                                                   \lambda & 0 & 0 \\
                                                   u'(x).(p(y,z)-y) & u(x).\frac{\partial p}{\partial y}(y,z)+1-u(x)  & u(x).\frac{\partial q}{\partial y}(y,z)\\
                                                   u'(x).(q(y,z)-z) & u(x).\frac{\partial p}{\partial z}(y,z) &  u(x).\frac{\partial q}{\partial z}(y,z)+1-u(x) 
                                                 \end{array}\right) $$
can be posed as 
 $$ D_{(x,y)}f=\left( \begin{array}{cc}
                                                   \lambda & 0 \\
                                                   u'(x).((p,q)(y,z)-(y,z)) & u(x). D_{(y,z)}(p,q) +(1-u(x))Id
                                               \end{array}\right)$$ if $(v_1,v_2,v_3)$ is considered as $(v_1, (v_2,v_3))$.
\end{rk}

\begin{rk}
  Since $ D \hat{f}_2 $ is very close to $Id_{\T^2}$, for all $x \in \T^3$ there exist invariant vectors $(0,v_2,v_3) \in T_x\T^3$ such that $\Vert D_xf(0,v_2,v_3)\Vert \sim \Vert (0,v_2,v_3)\Vert$. It yields that the central space of $f$ is two-dimensional.
\end{rk}
\begin{lema}\label{lowestiraf}
  For all $p \in \mathbb{T}^3$ and all $ v \in C^{u}_{\kappa} (p)$, $ \Vert D_{p}f (v) \Vert > 6 \Vert v \Vert $.
\end{lema}
\begin{proof}
Let $v=(v_{1},v_{2},v_3) \in C^{u}_{\kappa} (p)$ and recall $0< \kappa <3$ and $\lambda >20$, then $\left( \frac{\Vert D_{p}f (v) \Vert}{6.\Vert v \Vert}\right) ^{2} \geq \frac{(\lambda.||v_{1}||)^{2}}{36.(||v_1||^2+||(v_2,v_3)||^2)}  \geq  \frac{\lambda^2}{36\left(1+{\left(\frac{||(v_2,v_3)||}{||v_1||}\right)}^{2}\right)}>\frac{\lambda^2 }{36(1+\kappa^2)} >\frac{ 400}{ 360}>1.$ \end{proof}

\begin{rk}\label{lowkappachico}
  Since the definition of unstable cone is independent of the construction of $f$, $\kappa$ can be chosen small enough such that for all curves $\gamma$ satisfying that the derivative $\gamma' \subset C^u_{\kappa}(\gamma)$ at all times, then inradius and diameter of $\gamma$ can be identified. For the rest of the section assume that $\kappa$ is small enough so that this identification holds. Notice also that for all curves $\gamma$ with derivative inside the cone at all times and $ir(\gamma) \geq 2$, the projection of $\gamma$ to $S^1$ is surjective.
\end{rk}

\begin{clly}\label{lowdiametro}
 For all curves $\gamma$ such that for all $t$ where $\gamma$ is defined it holds that $\gamma'(t) \subset C^u_{\kappa}(\gamma(t))$, the diameter satisfies $diam(f(\gamma))> 6diam(\gamma)$. \hfill $\Box$
\end{clly}

\begin{clly}\label{lowconorobusto}
  There exists a $C^1$ neighborhood $\mathcal{U}_f$ of $f$ such that all $g$ in $\mathcal{U}_f$ admit an unstable cone of parameter $\kappa$, index 2 and vertex $p$ at every $p \in \T^3$ for which Corollary \ref{lowdiametro} holds. \hfill $\Box$
\end{clly}

We prove now that both the local stable and local unstable manifolds at $z$ are dense in $\T^3$. Since they are transversal, the {\it Inclination Lemma for endomorphisms} yields $C^1$ transitivity for $f$. Afterwards, we show that the proof is robust, so $f$ will also be $C^1$ robustly transitive.

\begin{lema}\label{lowumd}
  The local unstable manifold $W^u_{loc}(z)$ is dense in $\mathbb{T}^3$.
\end{lema}
\begin{proof} Let $V=V_1 \times V_2$ be any open set in $\mathbb{T}^3=S^1 \times \mathbb{T}^2$. We show that there exists a point in the unstable local manifold $W^u_{loc}(z)$ with a forward iterate in $V$.\\
Let $f$ be $f(x,y,z)=(\lambda x,f_2(x,y,z))$. Since $f_1(x)=\lambda x$ and $r\lambda >1$, $f(W^u_{loc}(z))\supset [-1,1] \times \{(1,1)\}$. In turn, $f(W^u_{loc}(z)) \cap K_3 \times \mathbb{T}^2 \neq \emptyset$.\\ Then, since $\mathcal{F}_1$ is a robustly minimal set, there exists $m \in \mathbb{N}$ such that $\langle \mathcal{F}_1 \rangle^m (1,1) \in V_2$. Let $a' \in V_1$ arbitrary and $\lambda^{-m}a' \in K_3$ (this is possible after item 1 at Remark \ref{lowrkdinf}). In turn, there exists $p=(\lambda^{-m-1}a',1,1)$ in $W^u_{loc}(z)$ such that $f^{1+n}(p)  \in V$. \end{proof}

\begin{rk}\label{cutsall} It follos from the proof that for all points such that $W^u(q)$ intersects $ \{0\}\times \T^2$, there exists a horizontal curve of length larger or equal to 2 that is contained in $W^u(q)$.
\end{rk}

\begin{lema}\label{lowsmd}
  The local stable manifold $W^s_{loc}(z)$ is dense in $\T^3$.
\end{lema}
\begin{proof}
Let $V=V_1 \times V_2$ be an open set in $\T^3=S^1 \times \T^2$ and let $W^s_{loc}(z)$ be $\{0\}\times B$ where $B=(1-r,1+r)^2$ is $W_{loc}^s(1,1)$ for $g_0$. Pick any point $p=(p_1,p_2,p_3) \in V_1 \times V_2$ and a well defined 'horizontal' curve $\gamma:(-s,s) \rightarrow V$ such that $\gamma(t)=p+t.\vec{e_1}=(p_1+t, p_2, p_3)$. Since for all $t$, $\gamma'(t)=\vec{e_1}$, it holds that $\gamma$ is a curve whose velocity lies inside the unstable cone field of $f$ at all times. By Corollary \ref{lowdiametro}, there exists $k_0 \in \N$ such that $diam(f^k(\gamma))\geq 6^k.diam(\gamma)>2$ for all $k \geq k_0$ (this shows that all future iterates of $\gamma$ project surjectively to $S^1$) which gives $f^{k_0}(\gamma) \cap \left( K_3 \times \T^{2} \right)\neq \emptyset$. Let $(\lambda^{k_0}.(p_1+t),f_2^{k_0}(p_2,p_3)) \in K_3 \times \mathbb{T}^2$. Then, since IFS$(\mathcal{F}_1)$ is strongly robustly minimal there exists $m \in \mathbb{N}$ such that $\langle \mathcal{F} \rangle^m(f^{k_0}_2(p_2,p_3)) \in B$ in $f_2^{k_0+m}(p_2,p_3)$. In turn, $f^{k_0+m}(\gamma) \supset S^1 \times \{ f_2^{k_0+m}(p_2,p_3)\}$ so $f^{k_0+m}(\gamma) \cap \{0\} \times B \neq \emptyset$. Therefore, it exists a point in $V$ (in $\gamma$) with a forward iterate in $W^s_{loc}(z)$. \end{proof}

\begin{thm}\label{lowfisRT}
  The map $f$ defined by Equation (\ref{lowmainmap}) is transitive.
\end{thm}
\begin{proof}
Apply the generalized inclination lemma for endomorphisms (\cite[Theorem 5.2]{hl}) together with lemmas \ref{lowumd} and \ref{lowsmd} to get the result. 
\end{proof}

We move on now on to the proof of robust transitivity.

\begin{lema}\label{T3tiene1NHinvariante}
There exists a $C^1$ neighborhood $\mathcal{U}_f$ of $f$ such that for all $G \in \mathcal{U}_f$: $G$ admits a unique $z'=(0,1,1)'$, a $G$-invariant submanifold $H_0$ of $\T^n$ and a diffeomorphism $h_1:\{0\} \times \T^2 \to H_0$ such that the restriction $G_{|H_0}$ is conjugated to $h_1 \circ g_0 \circ h_1^{-1}$.
\end{lema}
\begin{proof}\phantom{a}\\
Start considering $\mathcal{V}_{f}$ a $C^1$ neighborhood of $f$ that preserves the blending region $K \times \T^2$ in accordance with item 6 at remark  \ref{lowrkdinf}.  Reduce it until all maps in $\mathcal{V}_{f}$ admit unstable cones satisfying Corollary \ref{lowdiametro}.
Let $G$ be an arbitrary map in $\mathcal{V}_{f}$. After item 6 at Remark \ref{lowrkdinf}, $G$ admits a fixed saddle we call $z'=(0,1,1)'$ which is the continuation of the saddle $z$ of $f$. Recall that $f_{|{\lbrace 0 \rbrace} \times S^1}=g_1$ is Morse-Smale, so by their stability, there exists $\varepsilon >0$ such that for all $g: {\lbrace 0 \rbrace} \times \T^2 \to {\lbrace 0 \rbrace} \times \T^2$ satisfying $||g-g_1||<\varepsilon$, there exists a conjugating diffeomorphism $H_1$ such that $H_1 \circ g_1= g \circ H_1$.\\
As well,  $f \left( \left[ \frac{-r}{\lambda}, \frac{r}{ \lambda} \right] \times \T^2 \right) = K_0 \times \T^2$, the preimage map, is locally a diffeomorphism such that  $\bigcap_{n \in \mathbb{N}} f^{-n}(K_0 \times \T^2)={\lbrace 0 \rbrace}\times \T^2$, which implies $f^{-1}$ is normally hyperbolic at ${\lbrace 0 \rbrace}\times S^1$. Theorem   \ref{NHSP} yields that $\bigcap_{n \in \mathbb{N}} G^{-n}(K_0 \times \T^2):= N$ is a unique $G$-invariant submanifold of $\T^3$ with $z' \in N$ and that there exists a $C^1$ diffeomorphism $h_1: {\lbrace 0 \rbrace }\times  \T^2 \to N$ as close to the canonical inclusion as desired if $G$ is close enough to $f$. Adjust then $\mathcal{V}_f$,  such that $||g_1 - h_1^{-1}\circ G_{|N} \circ h_1||=||f_{|{\lbrace 0 \rbrace} \times \T^2} - h_1^{-1}\circ G_{|N} \circ h_1||< \varepsilon$ holds for all $G$.

But then, since it also exists a homeomorphism $H_1:  {\lbrace 0 \rbrace }\times \T^2 \to  {\lbrace 0 \rbrace }\times \T^2  $ such that $H_1 \circ g_1 =h_1^{-1}\circ G_{|N} \circ h_1 \circ H_1$ after the stabiltiy argument given for $g_1$, taking  $H= h_1 \circ H_1$, we conclude that $G_{|N}$ is conjugated to $h_1 \circ g_1 \circ h_1^{-1}$.   \hfill $\Box$
\end{proof}

Since all maps in  $\mathcal{F} $ are Morse-Smale, then:
\begin{clly}\label{T3tieneNHinvariantes}
 There exists a $C^1$ neighborhood $\mathcal{U}_f$ of $f$ such that for all $G \in \mathcal{U}_f$: $G$ six fixed hyperbolic points (continuated from $f$) and at every one of them, a normally hyperbolic submanifold (continuated from those of $f$) such that the dynamics are locally conjugated to those of $f$. \hfill $\Box$
\end{clly}

\begin{rk}\label{varperttransvT3}
We assume that for all $G$ in the neighborhood given by corollary \ref{T3tieneNHinvariantes}, it holds that the manifolds of the saddles of $G_{|H_3}$ y $G_{|H_5}$ project onto $\T^2$ intersecting transversally and with at most one common point, such that the only attractor of $G_{|H_5}$ projects contained in the basin of the only attractor of $G_{|H_3}$. This follows from the definition of $m_3$ y $m_5$ at the beginning of the section, along with the fact that the diffeomorphisms that send $\{p_i\} \times \T^2$ to $H_i$ are $C^1$ close to the canonical inclusions and that transversality is an open property. 
\end{rk}

\begin{lema}\label{GestirapatraT3}
For all $G$ in the  $C^1$ neighborhood $\mathcal{U}_f$ of $f$ given by corollary \ref{T3tieneNHinvariantes} it holds that for all open set $U$ of $\T^3$ and $U' \subset U$ any open subset, there exists $m_0 \in \N$ such that\footnote{$\pi$ refers to the projection onto the second factor $\T^2$.} the inradius $ir (\pi \circ G^{-m}(U')) \geq \frac {\sqrt{3}} {4}$ for all $m \geq m_0$.
\end{lema}
\begin{proof}
Let $\pi: \mathbb{T}^{n-1} \times S^1 \to S^1$ be the canonical projection. Notice that $||(\pi \circ G) - (\pi \circ f)||<||G-f||$. Then, by remark \ref{lowrkdinf} and definition \ref{lowmainmapita} we conclude that $(\pi \circ f)_{|K_0 \times S^1}={g}_1$ and that $(\pi \circ f)_{|K_1 \times S^1}={g}_2$. Since ${g}_1$ y ${g}_2$ are Morse-Smale, there exists a $C^1$ neighborhood $\mathcal{V}_f$ of $f$ such that $(\pi \circ G)_{|K_0 \times S^1}$ is conjugated to ${g}_1$ and $(\pi \circ G)_{|K_1 \times S^1}$ is conjugated to ${g}_2$, holding also for the local inverse $G^{-1}$. Let then $U$ be any open neighborhood of $z'$ and $U' \subset U$ any open subset. Recall that $G$ preserves the blending region $K \times \T^2$, hence $G^{-1}(U')$ contains a preimage in $K_0 \times \T^2$ and another one in $K_1 \times \T^2$. After the mentioned conjugations and lemma \ref{SRM}, $\pi(G^{-1}(U'))$ is a curve with a strictly larger length than $\pi(U')$ in at least one of these components. Let $X_1$ be the component with strictly larger lenght and consider $G^{-1}(X_1) \subset G^{-2}(U')$. Again, two preimages $G^{-1}(X_1)$ in $K_0 \times \T^2$ and $K_1 \times \T^2$ exist. Projecting onto $\T^2$, let $X_2$ be the one with strictly larger length. Keep going with the algorithm. If it stops, it means that the projections intersect all expanding regions, but as is seen in Lemma \ref{SRM}, this can only happen if the diameter of the projection is larger than $\frac{\sqrt{3}} {4}$. Therefore, there exists a number $m_0$ such that  $ir(\pi(G^{-m}(U'))\geq  \frac{\sqrt{3}} {4}$ for all $m \geq m_0$. 
\end{proof}

\begin{lema}\label{inestablescercadensasT3}
If $W^u(q,f)$ is dense in $\T^3$, then, there exists a $C^1$ neighborhood $\mathcal{V}_f$ of $f$ such that for all $G \in \mathcal{V}_f$ and for all open connected $V \subset \T^3$ such that $diam(\pi(V)) \geq \frac{ \sqrt{3}} {4}$, then $W^u(q',G) \cap V \neq \emptyset$.
\end{lema}

The idea is that, since $W^u(q,f)$ is dense, it cuts all open sets. Given any size, we can take a neighborhood of $f$ such that all perturbations intersect, say, the double of the given size. The formal proof follows.

\begin{proof}
Start recalling that by the condition stated over $\tau$ that defines $W^u(q,f)$, at remark \ref{lowrkdinf}, $f^p(W^u(q,f))$ projects surjectively over the first factor $S^1$ for all $p \geq 1$. This implies that for all iterates of $W^u(q,f)$, there exists a copy of the base factor $S^1$ coming from each of the $K_i$ (see figure \ref{obs54}). 

Let $\mathcal{C}$ be a grid on $\T^2$ comprised of squares of side lenght $\frac{1}{1000}$ and the family $\mathcal{C}=[-1,1]\times \mathcal{R}$. Since $W^u(q,f)$ is dense in $\T^3$, there exists a natural $n_i$ such that $f^{n_i}(W^u(q,f))$ intersects the interior of the $i$-th horizontal box of the grid $\mathcal{C}$. Define $N:=\max \{ n_i \}$ and choose $\varepsilon=\frac 1 {1000^2}$. Since the local unstable manifolds are $C^1$ close graphs in a given neighborhood, we set $\mathcal{V}_f$ such that for all  $G \in \mathcal{V}_f$, the map $h_G: f^N(W^u_{loc}(f)) \to G^N(W^u_{loc}(G))$ is a diffeomorphism satisfying $||h_G-Id||<\frac 1 {1000^2}$. 

Finally, let $V$ be an open set in  $\T^3$ with $diam(\pi(V))\geq \frac {\sqrt 3} {4}$. For this set, there exist $i_0 \in \N$ y $R_{i_0} \in \mathcal{C}$ such that  $R_{i_0} \subset [-1,1] \times \pi( V)$. For this $i_0$, there exists a point $x \in W^u_{loc}((0,1),f)$ such that $f^{n_{i_0}}(x) \in R_{i_0}\cap V$.  The map $h$ satisfies 
$G^{n_{i_0}}(h_G(x)) \in V_{i_0}$. Since $G^{n_{i_0}}(h_G(x))$ is $\frac {1} {1000^2}$ $C^1$ close to $f^{n_{i_0}}(x)$, then $G^{n_{i_0}}(h_G(x))$ contains a horizontal curve projecting surjectively over $S^1$, therefore $W^u(q',G)\cap V \neq \emptyset$.   \end{proof}

\begin{figure}[ht] 
\begin{center} 
\includegraphics[scale=0.185]{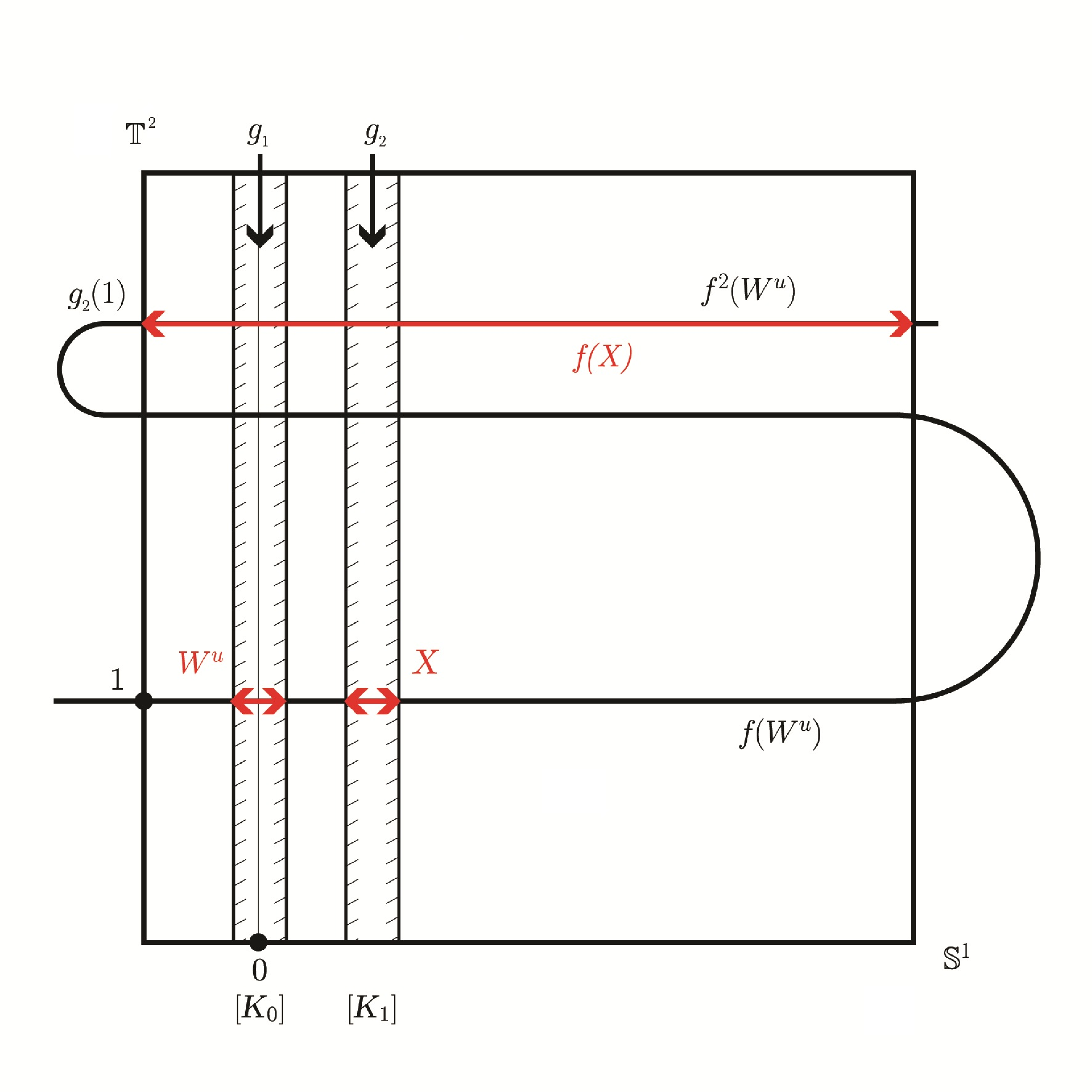}\caption{$W^u$ contains a dense foliation $f$-invariant of $\T^3$ comprised of horizontal curves that project surjectively.}\label{obs54} 
\end{center} 
\end{figure}

\begin{lema}\label{umdGT3} $W^u(q',G)$ is $G$-dense in $\mathbb{T}^3$. \end{lema} 
\begin{proof} Let $U$ be an open set in  $\mathbb{T}^3$. We see now that $W^u((0,1)',G)$ is dense proving that it intersects $U$. 
By Lemma \ref{GestirapatraT3}, given any open set $U'\subset U$, there exists $m_0$ such that $diam(\pi(G^{-m}(U'))) \geq \frac{\sqrt 3}{4}$ for all $m \geq m_0$. Lemma \ref{inestablescercadensasT3} shows that there exists $\mathcal{V}_f \in C^1$ such that for all $G \in \mathcal{V}_f$, the unstable set $W^u(q',G)$ intersects all open sets whose projections have an inradius of size at least $\frac {\sqrt 3}{4}$. From both, there exists $d>0$ such that $G^d(W^u_{loc}(q')) \cap G^{-m_0}(U') \neq \emptyset$. Since $U'\subset U$ it follows that  $ G^{d+m_0}(W^u_{loc}(q')) \cap U\neq \emptyset$ . \hfill $\Box$\end{proof} 

\begin{lema}\label{smdgT3} $W^s(q',G)$ is $G$-dense in $\mathbb{T}^3$. \end{lema}
\begin{proof} Let $U$ be any open set in $S^1 \times \mathbb{T}^2 $ and $\gamma:(-r,r) \to U$ a well defined curve such that $T \gamma \in C^u_{\kappa}(\gamma)$ the unstable cone field of $G$ at $\gamma$ at all times. Precisely, we can choose 
$p=(p_1, p_2, p_3) \in int(U)$ and $\gamma:(-r,r) \to U \ / \ \gamma(t)=(p_1+t, p_2, p_3)$. Corollary \ref{lowdiametro} yields that for all $m \in \mathbb{N}$, the inradius of the disk satisfies $ir_1(G ^{m}(\gamma))\geq 3^mir_{1}(\gamma)$. Since $diam(S^1)= 2$, there exists $m_0 \in \mathbb{N}$ such that for all $m \geq m_0$, $G^m(\gamma)$ projects surjectively over $S^1$. This implies that for all $m \geq m_0$,  $G^m(\gamma) \cap H_i \neq \emptyset, \forall i \in \{0,...,5\}$, being $\{H_i\}$ the family of invariant manifolds for $G$ given by Corollary \ref{T3tieneNHinvariantes}, which are transversal to  $G(W^u_{loc}(q',G))$ given the transversality of the invariant submanifolds of $f$ to $S^1$. Having stated the intersections exist, we conclude the Lemma discussing the position of the intersection points. 

If $G^{m_0}(\gamma) \cap H_3$ lies in the basin of attraction of $z'$ in  $G_{|H_3}$, then $G^{m_0+j}(\gamma) \cap W^s_{loc}(q') \neq \emptyset$ for some $j \in \mathbb{N}$ finishing the proof. If, on the contrary, $G^{m_0}(\gamma) \cap H_3$ is not in the basin of attraction of $z'$ in $G_{|H_3}$, then it either lies in the basin of attraction of the attractor of $G_{|H_5}$ or in the stable manifold of one of the saddles of $G_{|H_5}$. The position of the stable manifolds of $m_3$ and $m_5$ turns the second option into the first one after one iterate (see remark \ref{varperttransvT3}). Then, iterating $G^{m_0+1}(\gamma)$ following the point where it intersects $H_5$, there exists $d>0$ such that  $G^{m_0+1+d}(\gamma)$ is in the basin of attraction of the attractor of $G_{|H_5}$, again after remark \ref{varperttransvT3}. Finally, since $\gamma$ travels inside the unstable cone field and projects surjectively onto the first factor, we get that  $G^{m_0+1+d}(\gamma)$ cuts the basin of attraction of $z'$ in $G_{|H_3}$. \hfill  $\Box$  \end{proof}

\begin{thm}\label{lowfisRTT3} The map $f$ defined by equation (\ref{lowmainmap}) is robustly transitive. \end{thm} \begin{proof} 
Let $\mathcal{V}_f$ be the $C^1$ neighborhood of $f$ given by Lemma \ref{smdgT3}, contained in the neighborhood given by Lemma \ref{inestablescercadensasT3}. Since all $G \in \mathcal{V}_f$ satisfies Theorem \ref{lowfT3isT}, then $f$ is $C^1$ robustly transitive. \hfill  $\Box$  \end{proof}

\subsection{A singular endomorphism $F$ of $\T^3$.}
 Now that we have defined a robustly transitive endomorphism $f$ given by a blending region contained in $K \times \T^2$, we procceed to the second step of the construction by (robustly) artificially introducing critical points in the complement of $\tilde{K} \times \T^2$. The technique used to introduce the critical points is inspired by the construction carried on in \cite[Section 2.2]{ilp}. Once the surgery over $f$ is performed, a map $F$ satisfying the thesis of Theorem \ref{lowmain} arises.

\subsubsection{Construction of $F$}\label{lowconstructionofF}\phantom{.}\\
{\textit{Sketch of the construction:} We choose a point not in $\tilde{K} \times \T^2$ and set a ball centered at this point, inside the complement of $\tilde{K} \times \T^2$. By means of standard surgical procedures, we perturb $f$ to introduce a set of critical points inside the ball with the additional property that the resulting critical set is persistent. Since the surgery does not affect the blending region in $K \times \T^2$, the robust transitivity of the map $f$ defined by Equation (\ref{lowmainmap}) is inherited by the new map. We name the new map as $F$, and it satisfies the claim at Theorem \ref{lowmain}. \\

\begin{figure}[ht]
\begin{center}
\subfigure[]{\includegraphics[scale=1]{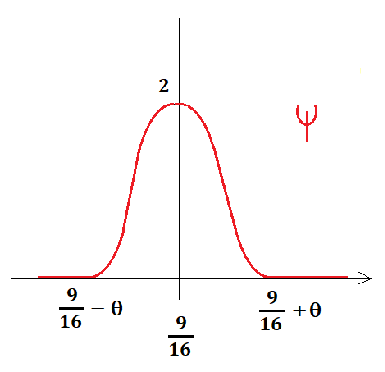}}
\subfigure[]{\includegraphics[scale=0.5]{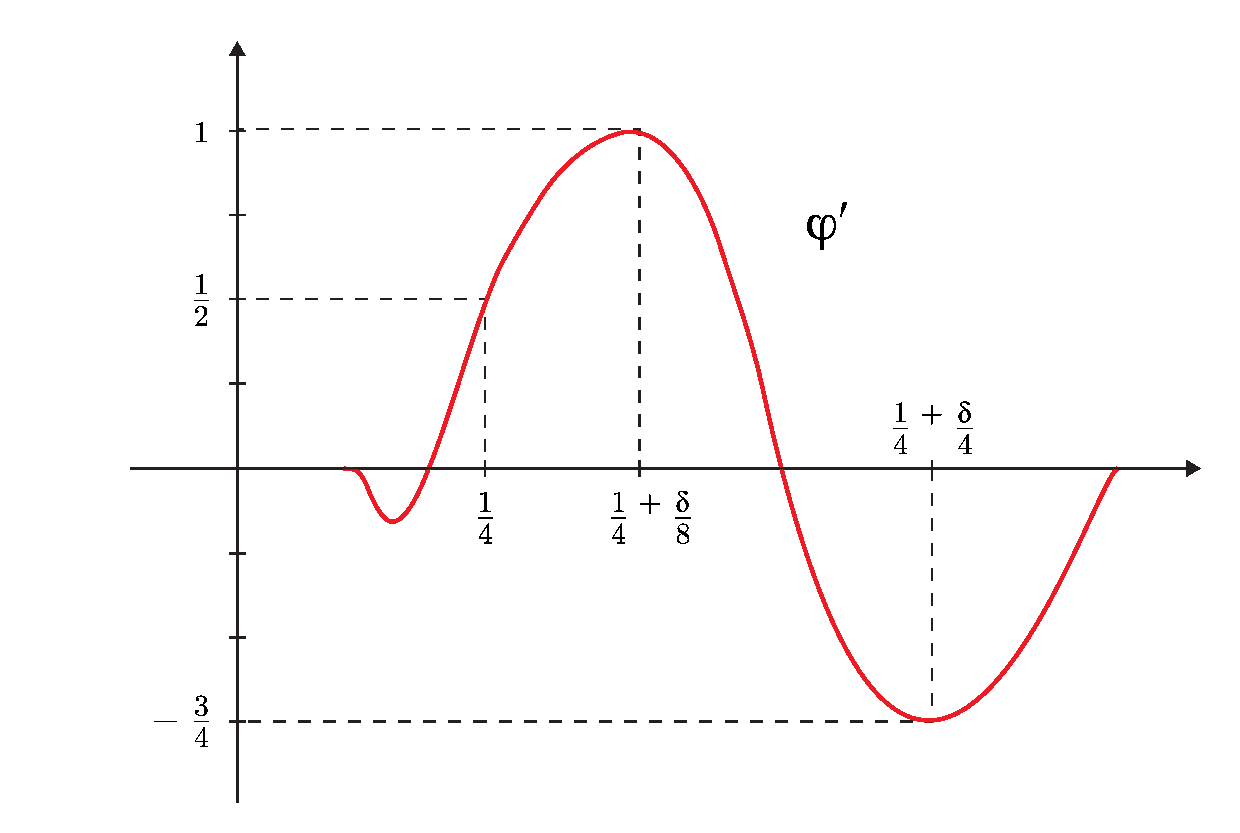}}
\caption{Graphs of $\psi$ and $\varphi'$ (following \cite{mo})}\label{lowfigura11}
\end{center}
\end{figure}

Let $p=\left(\frac{3}{4},0,\frac{1}{4} \right)\in (\tilde{K} \times \mathbb{T}^{3})^c$. Our goal is to define a ball of center $p$ to perform a perturbation in order to obtain the map $F$ we seek. To achieve this goal we need to fix a series of technical parameters; the choice to set all of them at the same time and at the beginning of the construction is in expectance of avoiding darkness and of that it will be clear how they depend on each other.\\

Start with $l>0$ satisfying that the ball $B_{(p,l)} \cap \left( \tilde{K} \times \T^{2} \right)= \emptyset$, this is possible since $p \notin \left( \tilde{K} \times \T^{2} \right)$. Fix a second parameter $\theta$ such that  $0< \theta  < \frac{l}{2}$  and define a smooth ($C^{\infty}$) function $\psi :\R\to\R$ with a unique critical point at $\frac{9}{16}$, with $\psi \left(\frac{9}{16}\right)=2$ and $\psi (x)=0$ for all $ x$ in the complement of $\left(\frac{9}{16}-\theta ,\frac{9}{16}+\theta  \right)$;  and an axis of symmetry in the line $x=\frac{9}{16}$ as shown in Figure \ref{lowfigura11} (a).

Set a last parameter $\delta $, with  $0<\delta < 2\theta $ verifying the following condition: since the derivative of $\psi$ is bounded once $\theta$ has been fixed, name the bound as $ m_\psi := m_\psi(\theta) = max  \{ | \psi'(x) |, x \in \R \} $ and impose on $\delta$ that $6.(1+\kappa).\delta.m_\psi.<\kappa$.  \\

Having fixed $\delta$, consider another smooth function $\varphi:\R\to\R$ such that:

\begin{itemize}
\item $\varphi' $ is as in Figure \ref{lowfigura11} (b).
\item $\frac{-3}{4}\leq \varphi'(x)\leq 1$ for all $x \in \R$. This gives $| \varphi'(x)| \leq 1$ for all $x \in \R$.
\item $\varphi' (x)=0$ for all
$ x \notin \left[\frac{1}{4}-\frac{\delta}{4},\frac{1}{4}+\frac{3\delta}{4} \right]$.
\item $\varphi' \left(\frac{1}{4}\right)=\frac{1}{2}$, $\varphi' \left(\frac{1}{4} +\frac{\delta}{8} \right)=1$, $\varphi' \left(\frac{1}{4} +\frac{\delta}{4}  \right)=-\frac{3}{4}$, $\varphi \left(\frac{1}{4}\right)=0$.

\end{itemize}
\begin{rk}
  $max \{ |\varphi (x)|: \ x \in \R \}\leq \delta.$
\end{rk}

We are now in condition to define a perturbation of $f$ in the direction of the last canonical vector $(0,0,1)$ that depends on $l$,$\theta$ and $\delta$ which by simplicity we call only $F$ and is defined as \begin{equation}\label{lowmapaefe}
           F_{l,\theta, \delta}:\T^3 \to \T^3 / F(x,y,z)= \left\{\begin{array}{c}
           \begin{aligned}
f(x,y,z) & \mbox{  if $(x,y,z) \notin B_{(p,l)} $} \\
(\lambda x, y , z - \varphi(z).\psi\left( x^2 +y^2 \right) &\mbox{  if $(x,y,z) \in B_{(p,l)} $}
\end{aligned}
\end{array}\right.
         \end{equation}

\begin{rk}\label{lowrkF}.

\begin{enumerate}
            \item For all $x \notin B_{\left(p,\frac{3\delta}{4}\right)}$ it holds that $F(x)=f(x)$.
            \item For all $x \notin \tilde{K}\times \T^{2}$ it holds that $f(x)=A(x)$.
          \end{enumerate}

\end{rk}
To make the reading easier we will denote $\varphi(z)$ as $\varphi$ and $\psi \left( x^2+y^2 \right)$ as $\psi$.

\begin{lema}\label{lowFps}

  The endomorphism $F$ defined by Equation (\ref{lowmapaefe}) is persistently singular.
\end{lema}
\begin{proof} Start computing the differential $D_{(x,y,z)} F$ with $(x,y,z) \in B_{(p,l)} $.

\begin{equation}\label{lowecuacion1}
D_xF =
\left( \begin{array}{ccccc}
\lambda & 0 & 0 \\
0 & 1 & 0 \\

-2.x.\varphi.\psi' & -2.y.\varphi.\psi'&  1-\varphi'.\psi \\
\end{array}
\right).
\end{equation}
Since the critical set of $F$ is defined as $S_F=\{x \in \T^{n} /det ( D_xF )=0\}$, Equation (\ref{lowecuacion1}) provides $det(D_xF)=\lambda (1-\varphi'.\psi) $. In turn, $S_F=\{x \in \T^{n} /\quad 1-\varphi'.\psi =0\rbrace.$\\
Notice that $S_F$ is nonempty since $p=\left(\frac{3}{4},0,\frac{1}{4}\right) \in S_F$. To prove that $S_F$ is persistent, consider the points $q_1=\left(\frac{3}{4},0, \frac{1}{4} + \frac{\delta}{4}\right)$ and $q_2=\left(\frac{3}{4},0, \frac{1}{4} + \frac{\delta}{8}\right)$ both in $B_{(p,l)}$. Evaluate determinants to obtain \textit{det}$(D_{q_1}F)= \frac{5\lambda}{2}$ and  \textit{det}$(D_{q_2}F)= -\lambda$. Therefore, for a neighborhood $\mathcal{U}_F \in C^{1}$ of radi $\frac{1}{2}$, every $ g \in \mathcal{U}_F$ satisfies $S_g\neq\emptyset$. \end{proof}

\subsubsection{Dynamics of $F$.}
We turn now to the last part of the article where we show that $F$ is $C^1$ robustly transitive. To prove it, observe first that after Remark \ref{lowrkF}, Lemma \ref{lowumd} holds for $F$ automatically. If we prove that Lemma \ref{lowsmd} also holds for $F$, then we can apply the same argument of Theorem \ref{lowfisRT} to $F$ to have the result. Notice that for Lemma \ref{lowsmd} to hold for $F$ we only need to show that $F$ admits an unstable cone $C_\kappa^u(x)$ at every point $x \in B_{(p,l)}$ which satisfies that for all curves $\gamma$ with $\gamma' \in C_\kappa^u(x)$ then $diam(F(\gamma))>6.diam(\gamma)$.

\begin{lema}\label{lowconoF1}
  For all $q \in B_{(p,l)}$ it holds that $C^{u}_{\kappa} (q)$ is an unstable cone for $F$.
\end{lema}

\begin{proof} From Equation (\ref{lowecuacion1}) we have for all $ v=(v_1,v_2,v_3) \in C^{u}_{\kappa} (x):$
$$D_{q}F (v) = (\lambda.v_{1},v_{2},-2.(xv_1+yv_2).\varphi.\psi' + v_3.(1-\varphi'.\psi)). $$

Call $ (u_1,u_2,u_3):=D_{q}F (v)$ and perform. We have:  $$ \frac{\Vert (u_{2},u_{3}) \Vert}{\vert u_{1}\vert}  = \frac{\Vert (v_{2}, -2.\langle (x,y),(v_1,v_2)\rangle.\varphi.\psi' + v_3.(1-\varphi'.\psi)) \Vert}{\vert \lambda.v_{1} \vert} \leq$$ $$ \leq \frac{|v_{2}|}{\vert \lambda.v_{1}\vert} + \frac{2.\Vert q \Vert.\Vert v \Vert.\vert \varphi \vert.\vert \psi' \vert}{\vert \lambda.v_{1} \vert} + \frac{\vert 1-\varphi' .\psi \vert. \vert v_{3}\vert}{\vert \lambda.v_{1} \vert} < \frac{\kappa}{\lambda} + 2.3.\left(\frac{1+\kappa}{\lambda}\right).\delta.m_\psi +\frac{3\kappa}{\lambda} < \kappa. $$ \\ Above, for the first inequality we use triangular and Cauchy-Schwarz; for the second one we use:
\begin{itemize}
\item $ v \in C^{u}_{\kappa} (x)$,
\item $ \Vert q \Vert < 3$,
\item $ \frac{\Vert v \Vert}{|\lambda.v_1|}\leq \frac{|v_1|+\Vert (v_2,v_3) \Vert}{|\lambda.v_1|}\leq \frac{1}{\lambda}+\frac{\kappa}{\lambda}$,
\item $ \vert \varphi \vert \leq \delta $,
\item $ \vert \psi' \vert < m_\psi $,
\item $ \vert 1-\varphi' .\psi \vert \leq 3 $ since $ \frac{-3}{4} \leq \varphi' \leq 1 $ and $ 0 \leq \psi \leq 2$.

\end{itemize}
And for the third one we use the condition $6.(1+\kappa).\delta.m_\psi.<\kappa$ imposed over $\delta$ and that $\lambda > 20$. \end{proof}

\begin{lema}\label{lowconoF2}
  For all $q \in B_{(p,l)}$ and all $ v \in C^{u}_{\kappa} (q)$ it holds that $ \Vert D_{q}F (v) \Vert > 6 \Vert v \Vert $.
\end{lema}
\begin{proof}
  Let $q \in B_{(p,l)}$ and $ (v_{1},v_{2},v_3) \in C^{u}_{\kappa} (q) \subset T_q \T^3$: $$\left( \frac{\Vert D_{q}F (v_1,v_2,v_3) \Vert}{6.\Vert (v_1,v_2,v_3) \Vert}\right) ^{2} \geq \frac{\lambda.v_{1}^{2}}{36.(v_1^2+||(v_2,v_3)||^2)}   \geq  \frac{\lambda^2}{36.\left(1+\frac{||(v_2,v_3)||^2}{|v_1|^2}\right)} > \frac{400 }{360}>  1. $$
\end{proof}

\begin{lema}\label{lowFrt}
  The map $F$ defined by Equation (\ref{lowmapaefe}) is $C^1$ robustly transitive.
\end{lema}
\begin{proof} From Lemmas \ref{lowconoF1} and \ref{lowconoF2} we conclude that Lemma \ref{lowsmd} holds for $F$. It was already mentioned that Lemma \ref{lowumd} holds for $F$. Consequently, Theorem \ref{lowfisRT} holds for $F$.
\end{proof}
We are now in condition to give the last proof of the section.

\begin{proof}[ \textbf{Proof of Theorem \ref{lowmain}}]
Define $\mathcal{U}_1 \in C^1$ an open neighborhood of $F$ where Lemma \ref{lowFps} holds and $\mathcal{U}_2 \in C^1$ an open neighborhood of $F$ where Lemma \ref{lowFrt} holds. Then, all maps belonging to $\mathcal{U}_F=\mathcal{U}_1 \cap \mathcal{U}_2$ are $C^1$ robustly transitive and have nonempty critical set and central spaces of dimension 2. \end{proof}

\section{A regular endomorphism $f$ of $\mathbb{T}^n$.} We move on to the proof of Theorems \ref{mainendo} and \ref{main}. Hoping that the readers got a hollow approach to the construction of the map in Section 3, we outline the ingredients for the construction for arbitrary $n$ and $k$. The proofs are mostly ommited, since they are in general an adaptation of the construction above. The readers who are curious about full proofs are gently invited to try them as excersises. For the perturbations that give Theorem \ref{main} from Theorem \ref{mainendo}, the reader can check \cite{mo2} since the techniques are exactly the same.

\subsection{Construction of $f$.} For the rest of the article, consider the $n$ dimensional torus $\T^n$ as the quotient $\R^n / [-1,1]^n$ endowed with the standard riemannian metric (euclidean). Fix $k$ integer such that $2 \leq k \leq n-1$ and $m=n-k$ and decompose $\mathbb{T}^n=\mathbb{T}^m \times \mathbb{T}^k$. Take the set $ \mathcal{F}_1 =\{g_{m+1},g_{m+2}\}$ to be given by Theorem \ref{SRM} for the second factor $\mathbb{T}^k$.
\begin{rk}\label{SRMhigh}\phantom{.}
  \begin{enumerate}
    \item $\max \{ ||Id_{\mathbb{T}^k}-g_{m+1}||, ||Id_{\mathbb{T}^k}-g_{m+2}||\} =: M \in \mathbb{R}$ exists.
\item As is stated in \cite{hn}, $\mathcal{F}_1$ can be constructed such that for given small $\delta >0$ the jacobians satisfy $( ||Dg_{m+1}|| -1)^2+( ||Dg_{m+2}||-1)^2<\delta$.
  \end{enumerate}
\end{rk}

Now that $\mathcal{F}$ is fixed, define a new Morse-Smale map $g_{m+3}$ such that it has an only attractor with a neighborhood contained in the basin of attraction of the attractor of $g_{m+1}$ and that the stable manifolds of its saddles intersect the stable manifolds of the saddles of $g_{m+1}$ transversally at most at one point. Request from $g_{m+3}$ to satisfy $||Id_{\mathbb{T}^k}-g_{m+3}||<M$. Afterwards, fix two real numbers $r$ and $\kappa$ such that $0<r< \frac{1}{10k}$ and $0< \kappa <3$. Let $\widehat{A} \in \mathcal{M}_n(\Z)$ be the diagonal matrix suggested below, with a large integer number $\lambda >> \max \{\frac{1}{r}, m, \frac{2M}{r\kappa}+2+\sqrt{n} \}$ (where $M$ is given by 1 of Remark \ref{SRMhigh}) in the $m$ first entries and in the other $k$ entries, $1$.

 \begin{equation}\label{matrizAgorro}
   \widehat{A}= \left(\begin{array}{cccccc}
 \lambda  & 0 & \cdots & \cdots & \cdots &0\\
0  &\ddots & \ddots & \cdots & \cdots & 0 \\
\vdots  & \ddots & \lambda & 0 & \cdots & 0 \\
\vdots  & \vdots &0 & 1 & \ddots & \vdots \\
\vdots  &  \vdots & \vdots & \ddots & \ddots & 0 \\
0  & \cdots & \cdots & 0 &0 &1\\
\end{array}
\right).
 \end{equation}

The matrix $\widehat{A}$ induces a regular endomorphism $A$ on the torus defined by
\begin{equation}\label{endoA}
  A:\T^n\rightarrow\T^n / \quad A(x_1,...,x_n)= (\lambda x_1,... , \lambda x_{n-k},x_{m+1},...,x_n).
\end{equation}

\begin{rk}\phantom{.}
  \begin{enumerate}
    \item  Since $r< \frac{1}{10k}$ then $\lambda > 10k>k+4$. Hence, $A$ display more than $(k+4)$ fixed points. In particular, all $\left(\frac{2z}{\lambda -1},...,\frac{2z}{\lambda -1},p_{n-k+1},...,p_n\right)\in \mathbb{T}^n$ where $z$ is integer is a fixed point of $A$.
        \item A very large value of $\lambda$ is chosen so that the unstable directions are strong.
         \item The map $A$ is $modulo$ $2$. All other maps of $\mathbb{T}^n$ to appear, likewise.
    \item The construction holds for any linear map in the isotopy class of maps with one eigenvalue of modulus larger than one since for all $\lambda_0$, there exists $p \in \mathbb{N}$ such that $\lambda_0^p > \lambda$.
  \end{enumerate}
\end{rk}

Let $r$ be the real number fixed at the beginning of the section and define $(k+4)$ disjoint subsets of $\T^m$ by $K_i=\left[\frac{2i}{\lambda -1}-r,\frac{2i}{\lambda +1}+r\right]^m$ and $K= \bigcup_{0 \leq i \leq k+2} K_i $. It is left as an excersise for the reader to verify that choosing $i$ as multiples of three is enough for the $K_i$ to be disjoint (since being $\lambda > 10k$, there is plenty of room for all these {\it cubes} to accomodate inside the torus).\\ Define next $\tilde{K}_{i}=\left[\frac{2i}{\lambda -1}-2r,\frac{2i}{\lambda -1}+2r \right]^m$ and $\tilde{K}= \bigcup_{0 \leq i \leq k+2}\tilde{K} $. For each $i$, define a smooth map $U_i: \R^m \rightarrow \R$ such that $U_i(x)=\prod_{j=0}^{m}u_i(x_j)$ where $u_i$ was defined in Section 3 (see Figure \ref{lowgraficou}). Let $U=\sum_{i=0}^{i=k+4}U_i$. It holds that $U_{|\tilde{K}^c}=0$ and $U_{|K}=1$. Notice also that $||\nabla U||:=\max \{ |\frac{\partial U}{\partial x_j}(x)|, x \in \mathbb{T}^m \} < \frac{2}{r}$.


Finally, let $\mathcal{F}_2=\{g_0,...,g_{n}\}$ be the family given by Lemma \ref{SRM} for the second factor $\T^{k}$, satisfying the properties claimed in the lemma for $\alpha \leq M$. Define
\begin{equation}
\hat{f}:\tilde{K} \times \T^{k} \rightarrow \T^n  /\quad  \hat{f}(x,y)=
                                                            (\lambda x , g_i(y)) \mbox{  if $x \in \tilde{K}_i$} \end{equation}
and extend $\hat{f}$ to
$f: \mathbb{T}^{m} \times \T^{k} \rightarrow \T^n $ by \begin{equation}\label{mainmap}
                                 f(x,y)=\left\{
                                \begin{array}{c}
                                 \begin{aligned}
 U(x).\hat{f}(x,y)+(1-U(x)).A(x,y)& \mbox{ \ \ \    if $x \in \tilde{K}$} \\
                                   A(x,y) & \mbox{ \ \ \  if $x \notin \tilde{K}$}
                               \end{aligned}
                                \end{array}\right.
                              \end{equation}

\begin{rk}\label{rkdinf}

 The following properties are straightforward to check:
\begin{enumerate}
  \item Calling $\hat{f}(x,y)=(\lambda x,\hat{f}_2(y))$, then $f(x,y)=(\lambda x,U(x).\hat{f}_2(y)+(1-U(x)).y)$.
   \item Since $\|Dg_i\| \leq \frac 3 2 $ for all $i$, together with Remark \ref{SRMhigh} give $\|D\hat{f}_2\|<2$.
  \item By construction of $f$, $\| Id -\hat{f}_2\| \leq M$.
  \item The restriction $f_{|\left(K \times \T^{k}\right)}=\hat{f}$.
  \item The restriction $f_{|{\left( \tilde{K} \times \T^{k}\right)}^c}=A$.

\end{enumerate}
\end{rk}

\subsection{Dynamics of $f$.}

 The most evident dynamical feature $f$ has is a strongly dominant expanding subspace along the first $m$ coordinates. It follows that there exists a family of unstable cones for $f$ in its direction. We make a pause here to check the existence of such an unstable cone field for $f$.

\begin{lema}\label{conosf}
  The map $f$ defined by Equation (\ref{mainmap}) admits an unstable cone of parameter $\kappa$, index $k$ and vertex $(x,y)$ at every $(x,y) \in \T^n$.
\end{lema}
\begin{proof}
The differential of $f$ at $(x,y)$ is given by $$
                                                D_{(x,y)}f=\left( \begin{array}{cc}
                                                   \lambda & 0 \\
                                                   \nabla U(x).(\hat{f}_2(y)-y) & U(x).D_y\hat{f}_2+(1-U(x)).y
                                                 \end{array}\right).
                                               $$ Then for all vectors $(v_1,v_2)$ of the tangent space of $\T^n$ at $(x,y)$ it is  $$D_{(x,y)}f(v_1,v_2)=\left( \begin{array}{c}
                                                   \lambda v_1 \\
                                                   \left[\nabla U(x).(\hat{f}_2(y)-y)\right]v_1 + \left[U(x).D_y\hat{f}_2+(1-U(x)).y\right]v_2
                                                 \end{array}\right).$$
\begin{rk}
Compare with remark \ref{iguales}.
\end{rk}
Consider now all vectors $(v_1,v_2)$ in $C^u_\kappa(x,y)$ and let $(w_1,w_2):=Df_{(x,y)}(v_1,v_2)$, we see that it is unstable by computing $$\frac{||w_2||}{||w_1||}=\frac{\left|\left|\left[\nabla U(x).(\hat{f}_2(y)-y)\right]v_1 + \left[U(x).D_y\hat{f}_2+(1-U(x)).y\right]v_2 \right|\right|}{\lambda||v_1||}\leq $$ $$ \leq \frac{||\nabla U||.||\hat{f}_2(y)-y||}{\lambda}+\frac{(|| U||.||D_y\hat{f}_2||+|1-U(x)|.||Id||).\kappa}{\lambda} \leq \frac{2M}{r\lambda}+\frac{(2+\sqrt{n})\kappa}{\lambda}<\kappa,  $$ \\ where in the first inequality we apply triangular and that $ (v_1,v_2) \in C^{u}_{\kappa} (x,y)$ and in the second and third inequalities we use:
\begin{itemize}
\item $ (v_1,v_2) \in C^{u}_{\kappa} (x,y)$,
\item $||\nabla U||.||\hat{f}_2(y)-y|| \leq \left(\frac{2}{r}\right)M$ by Remark \ref{rkdinf},
\item $ ||D\hat{f}_2||<2$ by Remark \ref{rkdinf},
\item $ \| Id\| \leq \sqrt{n} $,
\item $ \max_{x \in \R} \{|u(x)|,|1-u(x)|\}\leq 1$,
\item $\lambda > \frac{2M}{r\kappa}+2+\sqrt{n}$. \end{itemize}\end{proof}

\begin{lema}\label{estiraf}
  For all $ v \in C^{u}_{\kappa} (x,y)$ holds that $ \Vert D_{x}f (v) \Vert > 6 \Vert v \Vert $.
\end{lema}
\begin{proof} See Lemma \ref{lowestiraf}. \end{proof}

\begin{rk}\label{kappachico} Recall that if $B_k{(x,r)}$ denotes a ball of dimension $k$, the $k$-th. dimensional \textit{inradi} of a compact set $X \subset \mathbb{T}^n$ is $ir_k(X):=\max_{x \in X}\{r>0/ \quad B_k{(x,r)} \subset X \}$ and if $X$ is open the inradi is taken over its closure.\\
  Observe that if $X$ is a manifold and $Y$ an open subset such that $ir(Y) \geq diam(X)$ then $Y=X$.
\end{rk}

\begin{clly}\label{diametro}
  For all disks $\gamma$ such that for all $t$ where $\gamma$ is defined it holds that $T_t\gamma \subset C^u_{\kappa}(\gamma(t))$, the inradi satisfies $ir_k(f(\gamma))\geq 6ir_k(\gamma)$ for all $k \leq m$.
\end{clly}

\begin{clly}\label{conorobusto}
  There exists a $C^1$ neighborhood $\mathcal{U}_f$ of $f$ such that all $g$ in $\mathcal{U}_f$ admit an unstable cone of parameter $\kappa$, index $k$ and vertex $(x,y)$ at every $(x,y) \in \T^n$ for which Corollary \ref{diametro} holds.
\end{clly}

We highlight now some of the other relevant dynamical features the map $f$ possesses. All of them are straightforward to check. For the rest of the section, let $0_m=(0,0,..,0) \in \T^m$ and $1_k=(1,1,..,1) \in \T^k$.

\begin{rk}\phantom{.}
\begin{enumerate}
  \item For all $i$, $0 \leq i \leq n+2$, $f(K_i \times \T^k) = \T^3$. This yields $K \times \T^k$ gives rise to a protoblender structure on $\mathbb{T}^n$.
\item The point $(0_m,1_k)$ is a saddle fixed points of $f$, and the points $(0_m,0_k)$ is a repelling fixed point of $f$.
\item The local unstable manifold at $(0_m,1_k)$ is $W^u_{loc}(0_m,1_k)=(-r,r)\times\{1_k\}$.
\item The local stable manifold at $(0_m,1_k)$ is $W^u_{loc}(0_m,1_k)=\{0_m\} \times (1-r,1+r)^k$.
  \end{enumerate}
\end{rk}

For the rest of the article, when there is no risk of confusion, we write $(0,1)$ for $(0_m,1_k)$, $(0,0)$ for $(0_m,0_k)$, and so on. Next up, we prove that both the stable and unstable manifolds of $(0,1)$ are dense in $\T^n$. Since they are transversal, the {\it Inclination Lemma} will $C^1$ transitivity for $f$. We finish with a Theorem stating that both properties are robust.

\begin{lema}\label{umd}
  The unstable manifold $W^u(0,1)$ is dense in $\T^n$.
\end{lema}
\begin{proof}
  Repeat the arguments in the proof of Lemma \ref{lowumd}
\end{proof}

\begin{lema}\label{smd}
  The stable manifold $W^s(0,1)$ is dense in $\T^n$.
\end{lema}
\begin{proof}
  In the proof of Lemma \ref{lowsmd}, replace the horizontal curve $\gamma$ with a horizontal $m$-disk $\Gamma:(-s,s)^m \to V$ such that $\Gamma(t_1,...,t_m)=(p_1+t_1,...,p_m+t_m,p_{m+1},...,p_n)$ and repeat the arguments.
\end{proof}

\begin{thm}\label{fisRT}
  The map $f$ defined by Equation (\ref{mainmap}) is robustly transitive.
\end{thm}

\begin{proof}
  Follow the proof of Theorem \ref{lowfisRT} while seeing that all arguments hold in the high dimensional setting.
\end{proof}

\section{A singular endomorphism $F$ of $\T^n$.}
 Now that we have defined a robustly transitive endomorphism $f$ given by a blending region contained in $K \times \T^{k}$, we procceed to the second step of the construction by (robustly) artificially introducing critical points in the complement of $\tilde{K} \times \T^{k}$. Once the surgery over $f$ is performed, a map $F$ satisfying Theorem \ref{main} arises.

\subsection{Construction of $F$}\label{constructionofF} Everything that appears in the subsection to follow is just a slight adaptation of what was described in Section \ref{lowconstructionofF} with the purpose of making dimensions match. The readers are gently invited to try for themselves as an easy excersise before going into the details. Furthermore, if it is clear how such a perturbation is performed, the whole section can be skipped all the way.\\



Let $p=\left(\frac{-3}{4},0,...,0,\frac{1}{4}\right)\in\T^{n}$. Our goal is to define a ball of center $p$ to perform a perturbation in order to obtain the map $F$ we seek. The same parameters and functions defined in Section \ref{lowconstructionofF} are suitable for our objective. It is only needed to adjust the choice of $\delta >0$, while in our low dimensional example for $n=3$ we required $6.(1+\kappa).\delta.m_\psi.<\kappa$, in the general case we require $2n.(1+\kappa).\delta.m_\psi.<\kappa$.

For the perturbation, modify $f$ at Equation (\ref{mainmap}) in the direction of the last canonical vector $\vec{e_n}$ that depends on $l$,$\theta$ and $\delta$ which by simplicity we call only $F$ and is defined at $x=(x_1,...,x_n)$ as \begin{equation}\label{mapaefe}
           F_{r,\theta, \delta}:\T^n\to \T^n / F(x)= \left\{\begin{array}{c}
           \begin{aligned}
             f(x) &  \mbox{\ \ \   if $x \notin B_{(p,l)} $} \\
             A(x) - \varphi(x_n).\psi\left( \sum_{j=1}^{n-1}x_j^2 \right).\vec{e_n} & \mbox{\ \ \   if $x \in B_{(p,l)} $}
           \end{aligned}
 \end{array}\right.
         \end{equation}

\begin{rk}\label{rkF}\phantom{.}

\begin{enumerate}
            \item For all $x \notin B_{\left(p,\frac{3\delta}{4}\right)}$ it holds that $F(x)=f(x)$.
            \item For all $x \notin K^{\varepsilon}\times \T^k$ it holds that $f(x)=A(x)$.
          \end{enumerate}

\end{rk}
To make the reading easier we will denote $\varphi(x_n)$ as $\varphi$ and $\psi \left( \sum_{j=1} ^{n-1} x_j^2 \right)$ as $\psi$ omitting the evaluations appearing on the definition.

\begin{lema}\label{Fps}

  The endomorphism $F$ defined by Equation (\ref{mapaefe}) is persistently singular.
\end{lema}
\begin{proof}
Start computing the differential $D_x F$ at $x=(x_1,...,x_n) \in B_{(p,l)} $ to get

\begin{equation}\label{ecuacionn}
D_xF =
\left( \begin{array}{cccccccc}
\lambda Id_m   & 0 & \cdots &  0 & 0\\

  0   & 1 & 0 & 0 & \vdots \\
 \vdots      & 0 & 1 &  \cdots  & \vdots \\

0   & \vdots &  & & \vdots \\
0   & 0 & \cdots  & 1 & 0 \\
0   & -2.x_1.\varphi.\psi' & \cdots & -2.x_{n-1}.\varphi.\psi' & 1-\varphi'.\psi \\
\end{array}
\right)
\end{equation}
The proof is straightforward from the same calculations performed in Lemma \ref{lowFps}.
\end{proof}

\subsection{Dynamics of $F$.}
We can prove now that $F$ is $C^1$ robustly transitive. Observe first that after Remark \ref{rkF}, Lemma \ref{umd} holds for $F$ automatically. If we prove that Lemma \ref{smd} also holds for $F$, then we can apply the same reasoning of Theorem \ref{fisRT} to $F$ to have the result. Notice that for Lemma \ref{smd} to hold for $F$ we only need to show that $F$ admits an unstable cone $C_\kappa^u(x)$ at every point $x \in B_{(p,l)}$ which satisfies that for all disks $\gamma$ with $\gamma' \in C_\kappa^u(x)$ then $diam(F(\gamma))>6.diam(\gamma)$.

\begin{lema}\label{conoF1}
  For all $x \in B_{(p,l)}$, $C^{u}_{\kappa} (x)$ is an unstable cone of index $k$ for $F$.
\end{lema}
\begin{proof} 
See Lemma \ref{lowconoF1} for the idea and \cite[Lemma 4.4]{mo2} for the high dimensional setting. The condition $2n.(1+\kappa).\delta.m_\psi.<\kappa$ imposed over $\delta$ completes the proof. 
\end{proof}

\begin{lema}\label{conoF2}
  For all $x \in B_{(p,l)}$ and all $ v \in C^{u}_{\kappa} (x)$ it holds that $ \Vert D_{x}F (v) \Vert > 6 \Vert v \Vert $.
\end{lema}
\begin{proof}
See Lemma \ref{lowconoF2}.
\end{proof}

\begin{lema}\label{Frt}
  The map $F$ defined by Equation (\ref{mapaefe}) is $C^1$ robustly transitive.
\end{lema}
\begin{proof} From Lemmas \ref{conoF1} and \ref{conoF2} we conclude that Lemma \ref{smd} holds for $F$. It was already mentioned that Lemma \ref{umd} holds for $F$. Consequently, Theorem \ref{fisRT} holds for $F$. \end{proof}

We are ready to give the proof to the main Theorem of the article:

\begin{proof}[ \textbf{Proof of Theorem \ref{main}}]
Define $\mathcal{U}_1 \in C^1$ an open neighborhood of $F$ where Lemma \ref{Fps} holds and $\mathcal{U}_2 \in C^1$ an open neighborhood of $F$ where Lemma \ref{Frt} holds. Then, all maps belonging to $\mathcal{U}_F=\mathcal{U}_1 \cap \mathcal{U}_2$ are $C^1$ robustly transitive, have nonempty critical set and their central spaces are k-dimensional.
\end{proof}

\section{Final Remarks} The example exhibited in this article shows the existence of $C^1$ robustly transitive maps displaying critical points on any dimension with large spaces where hyperbolicity lacks. They are supported in tori of the form $\T^m \times \T^k$ for arbitrary $m$ and $k$, and lie in the isotopy class of the linear map $(\lambda x, y )$ where $\lambda$ is an integer (of absolute value) larger than one.
In the case where $k=2$, this result can be improved by constructing the family given by Theorem \ref{HNT} with the additional property of being $C^1$ close to the identity. This construction yields an improvement to the results above given by the following:

\begin{thm}
  Let $m \in \N$ and $\T=\T^m \times \T^2$. There exists then a map $f:\T \to \T$ which is $C^1$ close to $(\lambda x, y,z)$ and is $C^1$ robustly transitive while displaying critical points in a robust way.
\end{thm}

The author is certain about stating an analogous result for $\T=\T^m \times \T^3$ and believes that it also holds in $\T=\T^m \times \T^k$, but the proof will require certainly much more work with the techniques involved so far.

Furthermore, inspired by this and other results achieved before, the author poses the following:
\begin{conjecture}
  Let $M$ and $N$ be closed manifolds of arbitrary (finite) dimension. If $M$ supports an expanding map $F$, then $M \times N$ supports a $C^1$ robustly transitive map displaying robust singularities which is $C^1$-close to $(F,Id)$ .
\end{conjecture}
A partial affirmative answer, among skew products, has been proved by the author to hold.\\
Yet, many other open questions around this kind of maps and constructions remain. As already posed in \cite{mo2}, some of them are \textit{Would it possible to set this type of construction in manifolds that are not products? Is it possible that a fiber bundle (in stead of a product) admits a construction of this type? Would it be possible to carry on the proof starting from a matrix whose linear induced map belongs to a different isotopy class?}. Other questions appearing in the referred article have already found positive answers.

\section{Acknowledgments} The author wants to thank Prof. Jorge Iglesias for his encouragement on addressing the problem and the time invested on its discussion.  


\end{document}